\documentclass[a4paper]{article}

\usepackage{amssymb,amsfonts,amsmath,amsthm}
\usepackage[utf8]{inputenc}
\usepackage[babel=true,kerning,spacing]{microtype}
\usepackage{hyperref}
\usepackage{graphicx}
\usepackage{indentfirst}

\usepackage[american]{babel}

\usepackage[abbreviate=false]{biblatex}

\bibliography{document}

\DeclareFieldFormat{eprint:arxiv}{arXiv.org\addcolon{} \href{http://arxiv.org/abs/#1}{\nolinkurl{#1}}}
\DeclareFieldFormat{eprint:inria}{HAL-INRIA\addcolon{} \href{http://hal.inria.fr/inria-#1}{\nolinkurl{#1}}}
\DeclareFieldFormat{eprint:iacr}{IACR ePrint\addcolon{} \href{http://eprint.iacr.org/#1}{\nolinkurl{#1}}}

\usepackage[all]{xy}

\newtheorem{theorem}{Theorem}[section]
\newtheorem{proposition}[theorem]{Proposition}

\newtheorem{lemma}[theorem]{Lemma}
\newtheorem{algorithm}[theorem]{Algorithm}
\newtheorem{definition}[theorem]{Definition}
\newtheorem*{hypothesis}{Hypothesis}

\newenvironment{algo}{\begin{algorithm}~\nopagebreak

\noindent\begin{tabular}{rl}}{\end{tabular}\hfill\end{algorithm}}

\def\A{\mathcal A}
\def\B{\mathcal B}

\def\H{\mathcal H}

\def\O{\mathcal O}
\def\P{\mathcal P}

\def\FF{\mathbb F}
\def\NN{\mathbb N}
\def\QQ{\mathbb Q}

\def\ZZ{\mathbb Z}
\def\aaa{\mathfrak a}

\def\fff{\mathfrak f}

\def\ppp{\mathfrak p}

\def\rrr{\mathfrak r}
\def\sss{\mathfrak s}

\def\BBB{\mathfrak B}
\def\CCC{\mathfrak C}

\DeclareMathOperator{\disc}{disc}

\DeclareMathOperator{\norm}{N}

\DeclareMathOperator{\val}{val}
\DeclareMathOperator{\End}{End}
\DeclareMathOperator{\Gal}{Gal}

\DeclareMathOperator{\Pic}{Pic}

\title{Computing endomorphism rings of \\ abelian varieties of dimension two}

\author{Gaetan Bisson}
\date{\small University of French Polynesia}

\begin{document}

\maketitle

\begin{abstract}

Generalizing a method of Sutherland and the author for elliptic curves
\cite{endomorphism,grh-only} we design a subexponential algorithm for computing
the endomorphism rings of ordinary abelian varieties of dimension two over
finite fields. Although its correctness and complexity analysis rest on several
assumptions, we report on practical computations showing that it performs very
well and can easily handle previously intractable cases.

\paragraph{Note.} Some results of this paper previously appeared in the
author's thesis \cite{thesis}.

\end{abstract}

\section{Introduction}

Let $\A$ be an absolutely simple abelian variety of dimension $g$ defined over
a field with $q$ elements; its Frobenius endomorphism $\pi$ admits a
characteristic polynomial $\chi_\pi\in\ZZ[t]$ of which the $2g$ complex roots
have absolute value $\sqrt{q}$. Tate \cite{tate} shows that $\chi_\pi$ not only
encodes the cardinality of $\A$ over extension fields but also uniquely
identifies its isogeny class; Pila \cite{pila} later made this seminal result
effective by establishing that $\chi_\pi$ can be computed using polynomially
many elementary operations in $\log(q)$.

For principally polarized abelian varieties, the computation of isogenies can
also be done efficiently \cite{lubicz-robert} and is particularly relevant to
number theory and cryptography \cite{galbraith-isogenies}. In the generic case
where $\A$ is ordinary, the endomorphisms of $\A$ form a discrete subring of
maximal rank, \emph{an order}, $\End(\A)$ of $\QQ(\pi)$ that is stable under
complex conjugation and unchanged by base field extensions; this order
$\End(\A)$ is a finer invariant than $\chi_\pi$ better suited to
isogeny-related problems such as \cite{drew-hilbert}.

Its computation was first addressed by Kohel who obtained an exponential-time
method for ordinary elliptic curves \cite{kohel-phd}. This method was recently
improved by Sutherland and the author \cite{endomorphism} yielding an algorithm
of subexponential complexity under heuristic assumptions that were later shown
to follow from the generalized Riemann hypothesis \cite{grh-only}. While
Kohel's approach does not extend to dimension $g>1$
\cite[Example~8.3]{broker-gruenewald-lauter}, other exponential methods exist
for arbitrary $g$, namely those of Eisenträger and Lauter
\cite{eisentrager-lauter} and of Wagner \cite{wagner}.

This paper generalizes the techniques of \cite{endomorphism,grh-only} to
absolutely simple, ordinary, principally polarized abelian varieties of
dimension $g=2$ and obtains the first subexponential algorithm for computing
their endomorphism rings; its asymptotic complexity is
\[
L(q)^{g^2\sqrt{3}/2+o(1)}
\qquad\text{where}\qquad
L(q)=\exp\sqrt{\log(q)\cdot\log\log(q)}.
\]
Both its correctness and complexity bound rest on heuristic assumptions besides
the generalized Riemann hypothesis, and require the exclusion of a zero-density
set of worst-case varieties. Nevertheless, we find that it performs very well
on examples of moderate size which were previously intractable. Although most
of the techniques developped here apply to abelian varieties of arbitrary
dimension, we focus our analysis on the case $g=2$ which is of most interest to
cryptography.

Section~\ref{sec:cm} discusses the connection between isogenies and
endomorphisms. Sections~\ref{sec:isog} and \ref{sec:relat} then describe how to
compute isogenies and endomorphisms that allow us to exploit this connection,
and Section~\ref{sec:compare} puts this together into an algorithm for
comparing candidate rings; Section~\ref{sec:lattice} then explains how to
identify a lattice, from which the endomorphism ring is eventually recovered in
Section~\ref{sec:orders}. Finally, Section~\ref{sec:comp} reports on practical
computations.

\section{Isogenies and Endomorphism Rings}
\label{sec:cm}

We assume some familiarity with abelian varieties, isogenies, and endomorphism
rings; we refer to \cite[Chapter~V]{cornell-silverman} for background material
and to \cite{shimura-taniyama} for complex multiplication.

\bigskip

Again, consider an absolutely simple, ordinary, principally polarized abelian
variety $\A$ of dimension $g$ over a field with $q$ elements, and fix an
isomorphism of its endomorphism algebra $\QQ(\pi)=\QQ\otimes\End(\A)$ with a
number field $K$; this field is called the \emph{complex multiplication field}
of $\A$ and is a totally imaginary quadratic extension of a totally real number
field $K_0$ of degree $g$. Waterhouse \cite{waterhouse} shows that the
endomorphism rings of abelian varieties isogenous to $\A$ are exactly those
orders of $K$ stable under complex conjugation that contain
$\ZZ[\pi,\overline\pi]$, where $\overline\pi=q/\pi$; they form a finite lattice
(in the set-theoretic sense) with supremum the ring of integers $\O_K$.

Following Fouquet and Morain \cite{fouquet-morain}, we say that an isogeny
$\phi:\A\to\B$ is \emph{horizontal} when $\End(\A)$ and $\End(\B)$ are the same
order in $K$, and \emph{vertical} otherwise. In a sense, horizontal isogenies
are the prevalent case:

\begin{lemma}\label{lem:jump}
If $\phi:\A\to\B$ is an isogeny with kernel isomorphic to $(\ZZ/\ell\ZZ)^g$,
the index $[\End(\A)+\End(\B):\End(\A)\cap\End(\B)]$, which we call the
\emph{distance} between the orders $\End(\A)$ and $\End(\B)$, is a divisor of
$\ell^{2g-1}$.
\end{lemma}

\begin{proof}
Since $\phi$ splits the multiplication-by-$\ell$ map, we have
$\ell\End(\A)\subset\End(\B)$ and, the latter being an order, we further have
$\ZZ+\ell\End(\A)\subset\End(\B)$; we thus obtain the lattice of
Figure~\ref{fig:distance}.
\begin{figure}
\[\xymatrix@C=-42pt{
 & \End(\A)+\End(\B) \ar@{-}[dl]_{a} \ar@{-}[dr]^{b} & \\
\End(\A) \ar@{-}[dr]_{b} & & \End(\B) \ar@{-}[dl]^{a} \\
 & \End(\A)\cap\End(\B) \ar@{-}[d]^c & \\
 & \ZZ+\ell\End(\A)+\ell\End(\B) \ar@{-}[dl]_{d} \ar@{-}[dr]^{e} & \\
\ZZ+\ell\End(\A) \ar@{-}[dr]_{e} & & \ZZ+\ell\End(\B) \ar@{-}[dl]^{d} \\
 & \ZZ+\ell\End(\A)\cap\ell\End(\B)
}\]
\caption{Lattice of orders for an isogeny $\A\to\B$ of kernel $(\ZZ/\ell\ZZ)^g$.}
\label{fig:distance}
\end{figure}
As they are indices of the form $[\O:\ZZ+\ell\O]$,
the products $bcd$, $ace$, and $cde$ are all equal to $\ell^{2g-1}$ which
implies $bcd\cdot ace/cde=\ell^{2g-1}$ and hence $ab=\ell^{2g-1}/c$.
\end{proof}

Since the distance between endomorphism rings of isogenous abelian varieties
divides the index $[\O_K:\ZZ[\pi,\overline\pi]]$, vertical isogenies only exist
for finitely many primes $\ell$. On the other hand, horizontal isogenies occur
for a positive density of primes $\ell$ and their graph structure is relatively
well understood: as we will see, they form a Cayley graph for the following
group.

\begin{definition}\label{def:ccc}
For any order $\O$ in a complex multiplication field $K$, denote by $I_\O$ the
group consisting of all pairs $(\aaa,\rho)$ satisfying
$\aaa\overline\aaa=\rho\O$ with $\aaa$ an invertible fractional ideal of $\O$
and $\rho$ a totally positive element of $K_0$, endowed with component-wise
multiplication; also, let $P_\O$ be its subgroup formed by pairs of the form
$(\mu\O,\mu\overline\mu)$ for $\mu\in K$. The quotient group $I_\O/P_\O$ is
called the \emph{polarized class group} of $\O$ and is denoted by $\CCC(\O)$.
\end{definition}

Note that this group is unchanged if we additionally require that $\aaa$ (and
$\mu$) be coprime to a fixed integer $\nu$; to allow us to compare ideals of
$I_\O$ as $\O$ varies, we will from now on only consider class representatives
of this type with $\nu=\disc(\ZZ[\pi,\overline\pi])$. For maximal orders, the
theorem below shows that such elements correspond to horizontal isogenies.

\def\stuff{\cite[§14]{shimura-taniyama}}
\begin{theorem}[\stuff]\label{th:shimura}
When $\End(\A)$ is maximal, one can associate a horizontal isogeny of degree
$\norm_{K/\QQ}(\aaa)$ to every $(\aaa,\rho)\in I_{\End(\A)}$ with $\aaa$
coprime to the characteristic, so as to induce a free action of
$\CCC(\End(\A))$ on the isogeny class of $\A$ up to isomorphisms.
\end{theorem}

The corresponding result for non-polarized abelian varieties, which sees the
polarized class group replaced by the classical ideal class group, also holds
for arbitrary orders \cites[§7]{shimura-taniyama}[§7]{waterhouse}. While only
the case of maximal endomorphism rings has been treated in the literature, the
above theorem is expected to hold for arbitrary orders too, and we assume that
it does as one of our heuristics in this paper.

\begin{hypothesis}[A]
Theorem~\ref{th:shimura} holds even if $\End(\A)$ is not maximal.
\end{hypothesis}

\section{Evaluating Isogenies}
\label{sec:isog}

Until recently, isogenies could only be efficiently evaluated for elliptic
curves \cite{velu} and special families in higher dimension \cite{bost-mestre}.
The work of Lubicz and Robert \cite{lubicz-robert} has made it possible to
efficiently compute general, polarization-preserving isogenies in dimension
$g>1$; more precisely, we have:

\def\stuff{\cite[Theorem~1.2]{cosset-robert}}
\begin{proposition}[\stuff]\label{prop:isog}
Let $\H$ be a rational isotropic subgroup isomorphic to $(\ZZ/\ell\ZZ)^g$ of an
abelian variety of dimension $g=2$ with $\ell$ a prime different from the
characteristic. The separable isogeny with kernel $\H$ can be evaluated with a
worst-case complexity of $\ell^{3g+o(1)}$ operations in the base field.
\end{proposition}

The isogenies we are concerned with correspond to elements $(\aaa,\ell)$ of
$\CCC(\ZZ[\pi,\overline\pi])$; to evaluate them, we first need to identify the
kernel $\H$ corresponding to a given $(\aaa,\ell)$. Note that we can write
$\aaa=\ell\O+f(\pi)\O$ for some factor $f$ of $\chi_\pi\bmod\ell$. Now we take
$\H$ to be the subgroup of $\A[\ell]$ on which the Frobenius acts with
characteristic polynomial $f$; it is unique since we restrict to ideals $\aaa$
coprime to $\nu=\disc(\ZZ[\pi,\overline\pi])$. In effect, this identification
fixes an isomorphism between $\QQ\otimes\End(\A)$ and the complex
multiplication field $K$ (mapping a fixed root of $\chi_\pi$ to the Frobenius
endomorphism) as was required in Section~\ref{sec:cm}, and it only matters that
this be done consistently within a given isogeny class.

Points of $\H$ are defined over an extension field whose degree is the
multiplicative order of $x$ in $\ZZ[x]/(f)/(\ell)$, that is, at most
$\norm_{K/\QQ}(\aaa)-1$. Over that extension, an algorithm of Couveignes
\cite[§8]{couveignes-torsion} may be used to compute the $\ell$-torsion
subgroup of $\A$ assuming that points of $\A$ can be drawn uniformly at random.
This can be done efficiently for Jacobian varieties by using the underlying
curve.

\bigskip

The algorithm implicitly referred to by this Theorem returns a representative
of the isogenous isomorphism class $\A/\H$ defined over the field of definition
of individual points of $\H$, even if $\H$ itself and therefore $\A/\H$ are
rational. For abelian varieties of dimension $g=2$ represented as Jacobians of
genus-two curves, one can use a method of Mestre \cite{mestre} to find, after
each isogeny evaluation step, a representative of the isomorphism class $\A/\H$
defined over the minimal field. See \cite{cosset-robert,avisogenies} for
details.

Abelian varieties of arbitrary dimension may be represented by theta constants.
The points on such varieties are the roots of the Riemann equations, and can
therefore be drawn quasi-uniformly at random by solving this system using a
Gröbner basis algorithm, once enough variables are specialized to random values
so that its rank is full. Isomorphism classes of abelian varieties correspond
to orbits of theta constants under the action of the symplectic group, and can
thus also be identified efficiently. From now on, we will however focus on the
case where $g=2$ and hence restrict to abelian varieties given as Jacobians of
hyperelliptic curves.

\section{Generating Short Relations}
\label{sec:relat}

First recall a well-known elementary result about (polarized) class groups.

\begin{lemma}\label{lem:subccc}
For any two orders $\O\subset\O'$ containing $\ZZ[\pi,\overline\pi]$, the map
$(\aaa,\rho)\in I_\O\to(\aaa\O',\rho)\in I_{\O'}$ induces a natural morphism of
polarized class groups $\CCC(\O)\to\CCC(\O')$; this morphism is surjective when
restricted and corestricted to elements satisfying $\rho\in\QQ$.
\end{lemma}

The above result allows us to compare the polarized class group $\CCC(\O)$ as
the order $\O$ varies. More explicitly, we use \emph{relations} to characterize
polarized class groups.

\begin{definition}
We call a \emph{relation} any tuple $(\alpha_1,\ldots,\alpha_k)$ of elements of
$\CCC(\ZZ[\pi,\overline\pi])$. We say that this relation \emph{holds in $\O$}
if the product $\alpha_1\cdots\alpha_k$ is trivial in $\CCC(\O)$ through the
map of the above lemma, and that it \emph{holds in $\A$} if the corresponding
isogeny chain $\phi_{\alpha_1}\circ\cdots\circ\phi_{\alpha_k}$ maps $\A$ to an
isomorphic abelian variety.
\end{definition}

By Theorem~\ref{th:shimura}, if every relation that holds in $\O$ also does in
$\A$, the group $\CCC(\End(\A))$ must be a quotient of $\CCC(\O)$, and we will
later see that this implies that $\O\subset\End(\A)$ locally at almost all
prime.

The computation of class groups of algebraic orders is a classical topic that
has led to the development of advanced algorithms for generating ideal
relations. However, to effectively compare $\CCC(\O)$ with $\CCC(\End(\A))$ we
require that the isogenies corresponding to the relations we select be
computable in reasonable time; this places two additional constraints:
\begin{itemize}
\item Elements $(\alpha_i)$ of our relations must correspond to maximal isotropic isogenies.
\item Their number $k$ and norms $(\norm_{K/\QQ}(\alpha_i))$ must be bounded.
\end{itemize}
The latter constraint is already addressed in \cite[§6]{grh-only} whose results
and proofs carry directly over to orders in CM-fields of arbitrary degree. Let
us now explain how to additionally satisfy the former.

\bigskip

Let $\Phi$ be a type for $K$, that is, a set of representatives for embeddings
of $K$ into its normal closure $K^c$ up to complex conjugation. Its \emph{type
norm}
\[
\norm_\Phi:x\longmapsto\prod_{\phi\in\Phi}\phi(x)
\]
maps $K$ to its reflex field $K^r$, the fixed field of
$\{\sigma\in\Gal(K^c/\QQ):\sigma\Phi=\Phi\}$, and induces a morphism taking
ideals $\aaa$ of $K$ to elements $(\norm_\Phi(\aaa),\norm_{K/\QQ}(\aaa))$ of
$\CCC(\O^r)$ for any order $\O^r$ of $K^r$ with discriminant coprime to $\nu$.
Types of absolutely simple abelian varieties are primitive, which implies that
${K^r}^r=K$; hence the type norm of the reflex type $\Phi^r$, the restriction
to $K^r$ of inverses of automorphisms of $K^c$ induced by $\Phi$, or
\emph{reflex type norm}, maps ideals of $K^r$ to $\CCC(\O)$ for any order $\O$
containing $\ZZ[\pi,\overline\pi]$. See Figure~\ref{fig:typenorm}.

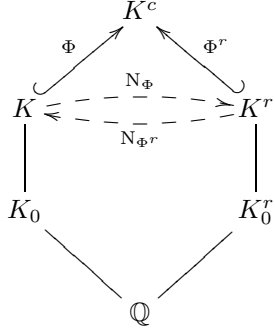
\begin{figure}
\begin{center}
\[\xymatrix{
 & K^c \ar@{<-_{)}}[dl]_{\Phi} \ar@{<-^{)}}[dr]^{\Phi^r} & \\
K \ar@{-}[d] \ar@{-->}@/^5pt/[rr]^{\norm_\Phi} & & K^r \ar@{-}[d] \ar@{-->}@/^5pt/[ll]^{\norm_{\Phi^r}} \\
K_0 \ar@{-}[dr] & & K^r_0 \ar@{-}[dl] \\
 & \QQ & \\
}\]
\caption{The complex multiplication field, its reflex field, and type norm maps.}
\label{fig:typenorm}
\end{center}
\end{figure}

Images through $\norm_{\Phi^r}$ of prime ideals of $K^r$ are polarized ideals
of which the corresponding isogenies can be computed using
Proposition~\ref{prop:isog}. Therefore, to obtain relations that hold in a
given order $\O$ and whose corresponding isogenies can be efficiently
evaluated, we first generate tuples of ideals $(\aaa_i)$ of $\O$ whose product
is principal using the method of Buchmann \cite{buchmann} as modified in
\cite[§6]{grh-only}, and then take its image through
$\norm_{\Phi^r}\circ\norm_\Phi$; the total norm of the resulting relation is
$\sum_i\norm_{K/\QQ}(\aaa_i)^{g^2}$. Formally, we obtain:

\begin{algo}\label{alg:relat}
   \textsc{Input:} & An order $\O$ and a parameter $\gamma>0$.
\\ \textsc{Output:} & A relation holding in $\O$ of which the corresponding isogeny can be computed efficiently.
\smallskip
\\ 1. & Form the set $\BBB$ of prime ideals $\ppp$ of $\O$ with norm less than $L(\disc(\O))^\gamma$.
\\ 2. & Draw a vector $x\in\ZZ^\BBB$ uniformly at random with coordinates
\\    & $|x_\ppp|<\log(\disc(\O))^{4+\epsilon}$ when $\norm_{K/\QQ}(\ppp)<\log(\disc(\O))^{2+\epsilon}$ and $x_\ppp=0$ otherwise.
\\ 3. & Compute the reduced ideal representative $\aaa$ of $\prod \ppp^{x_\ppp}$.
\\ 4. & If $\aaa$ factors over $\BBB$ as $\prod \ppp^{y_\ppp}$:
\\ 5. & \qquad Return the relation containing $\norm_{\Phi^r}(\norm_\Phi(\ppp))$ with multiplicity $x_\ppp-y_\ppp$ for $\ppp\in\BBB$.
\\ 6. & Go back to Step~2.
\end{algo}

For details on Step~4, and more generally on computing ideal relations in
number fields, we refer to \cite{cohen-diaz-olivier}. All steps above have
previously been analyzed except for the evaluation of
$\norm_{\Phi^r}\circ\norm_\Phi$ which only uses polynomial time; we therefore
borrow the assumptions and complexity bound of \cite{buchmann} for
Algorithm~\ref{alg:relat}:

\begin{hypothesis}[B]
The generalized Riemann hypothesis holds, and reduced ideals have the
smoothness properties of integers of comparable size.
\end{hypothesis}

\begin{proposition}\label{prop:relat}
Under Hypothesis (B), this algorithm generates a relation with total norm
$L(\disc(\O))^{g^2\gamma+o(1)}$ in expected time
$L(\disc(\O))^{\gamma+o(1)}+L(\disc(\O))^{1/(4\gamma)+o(1)}$.
\end{proposition}

\section{Comparing candidate endomorphism rings}\label{sec:compare}

Our main idea to compute $\End(\A)$ is to exploit Theorem~\ref{th:shimura}: we
compare the structure of polarized class groups of candidate endomorphism rings
$\O$ with that of isogenies from the variety $\A$. For this, we generate
relations that hold in $\O$ using Algorithm~\ref{alg:relat} and test whether
the corresponding isogenies map to isomorphic varieties.

It is important to observe that Algorithm~\ref{alg:relat} only outputs
relations whose elements lie in the image of $\norm_{\Phi^r}\circ\norm_\Phi$;
these may not be all the relations that hold in $\O$, but only a sublattice
$\Lambda_\Phi(\O)$ of them. We start by computing this lattice and, for this,
use the following algorithm.

\begin{algo}\label{alg:test}
   \textsc{Input:} & An absolutely simple, ordinary, principally polarized abelian variety $\A$
\\ & of dimension $g$ defined over $\FF_q$ and an order $\O$ containing $\ZZ[\pi,\overline\pi]$.
\\ \textsc{Output:} & Whether $\Lambda_\Phi(\O)\subset\Lambda_\Phi(\End(\A))$.
\smallskip
\\ 1. & Repeat $5g^2\log_2(q)$ times:
\\ 2. & \qquad Find a relation $(\alpha_1,\ldots,\alpha_k)$ of $\CCC(\O)$ using Algorithm~\ref{alg:relat}.
\\ 3. & \qquad If $\phi_{\alpha_1}\circ\cdots\circ\phi_{\alpha_k}$ does not map $\A$ to an isomorphic variety, return false.
\\ 4. & Return true.
\end{algo}

If $\O'\subset\O$ are two orders containing $\ZZ[\pi,\overline\pi]$ stable
under complex conjugation, it follows from Hypothesis (B) that the relations
generated by Algorithm~\ref{alg:relat} are quasi-uniformly distributed in the
quotient $\Lambda_\Phi(\O)/\Lambda_\Phi(\O')$; see \cite[§6]{grh-only} for a
proof stated for imaginary quadratic fields but which readily carries over to
arbitrary CM-fields. Therefore, the relations output after $5g^2\log_2(q)$ runs
of Algorithm~\ref{alg:relat} characterize $\Lambda_\Phi(\O)$ with error
probability at most $(1/2)^{5g^2\log_2(q)}=1/q^{5g^2}$.

To balance the cost of generating relations using Algorithm~\ref{alg:relat}
with that of evaluating the corresponding isogenies, we set
$\gamma=1/(2g\sqrt{3})$ in Step 2, and obtain the following result.

\begin{proposition}\label{prop:test}
Under Hypotheses (A,B), Algorithm~\ref{alg:test} determines whether the
relation lattice $\Lambda_\Phi$ of $\End(\A)$ contains that of a prescribed
order $\O$ with error probability $1/q^{5g^2}$ using an expected
\[
L(|\disc(\O)|)^{g\sqrt{3}/2+o(1)}
\]
operations in the base field.
\end{proposition}

\paragraph{Note.} Rather than generating independent relations for each order
$\O$ of the lattice to be tested, one might be tempted to first compute the
full class group structure of the maximal order $\O_K$ and then deduce
relations of smaller orders $\O$ via the exact sequence:
\[
1\to\O^\times\to\O_K^\times\to(\O_K/\fff)^\times/(\O/\fff)^\times\to\Pic(\O)\to\Pic(\O_K)\to 1
\]
where $\fff$ is the conductor of $\O$, that is, the largest ideal of both $\O$
and $\O_K$. This has two disadvantages: first, computing class groups is much
more expensive than generating just $O(\log(q))$ relations; second, the
relations of $\O$ given directly by the exact sequence above grow linearly in
the index $[\O_K:\O]$, and deriving subexponential-size relations requires
using an algorithm similar to \ref{alg:relat} anyhow.

\section{Locating the relation lattice}\label{sec:lattice}
\label{sec:almost}

Before identifying orders $\O$ satisfying
$\Lambda_\Phi(\O)=\Lambda_\Phi(\End(\A))$, let us first bound the number of
candidates, that is, orders containing $\ZZ[\pi,\overline\pi]$ stable under
complex conjugation, and their discriminants.

\begin{lemma}\label{lem:ord-bound}
We have:
\[
\left|\disc(\ZZ[\pi,\overline\pi])\right|<4^{g(2g-1)}q^{g^2},
\]
\[
\left[\O_K:\ZZ[\pi,\overline\pi]\right]<2^{g(2g-1)}q^{g^2/2}.
\]
\end{lemma}

\begin{proof}
All $2g$ complex roots of $\chi_\pi$ have absolute value $\sqrt{q}$, so we have
$|\disc(\chi_\pi)|<(2\sqrt{q})^{2g(2g-1)}$. The bounds then follow from the
classical relation $[\O:\O']^2=\disc(\O')/\disc(\O)$ and, for the first one,
the identity $[\ZZ[\pi,\overline\pi]:\ZZ[\pi]]=q^{g(g-1)/2}$ and, for the second one,
the triviality $|\disc(\O_K)|>1$.
\end{proof}

These bounds are nearly tight so there might be exponentially many candidate
endomorphism rings; to efficiently identify $\Lambda_\Phi(\End(\A))$ among
them, we exploit the identity
$\O\subset\O'\Rightarrow\Lambda_\Phi(\O)\subset\Lambda_\Phi(\O')$ by performing
an $n$-ary search in the lattice of orders using the following algorithm
adapted from \cite{grh-only}.

\begin{algo}\label{alg:ascend}
   \textsc{Input:} & An absolutely simple, ordinary, principally polarized abelian variety $\A$
\\ & of dimension $g$ defined over a field with $q$ elements.
\\ \textsc{Output:} & The relation lattice of its endomorphism ring.
\smallskip
\\ 1. & Compute the Frobenius polynomial $\chi_\pi$ of $\A$.
\\ 2. & Factor its discriminant and construct the order $\O'=\ZZ[\pi,\overline\pi]$.
\\ 3. & For orders $\O$ directly above $\O'$:
\\ 4. & \qquad If $\Lambda_\Phi(\O)\subset\Lambda_\Phi(\End(\A))$, set $\O'\leftarrow\O$ and go to Step~3.
\\ 6. & Return $\Lambda_\Phi(\O')$.
\end{algo}

For Step~2, we use the unconditional factoring method of Lenstra and Pomerance
\cite{lenstra-pomerance}; its complexity is $L(|\disc(\chi_\pi)|)^{1+o(1)}$,
that is, at most $L(q)^{g\sqrt{2}+o(1)}$. Alternatively, one may rely on the
number field sieve \cite{nfs} which has a heuristically better runtime.

By \emph{directly above}, we mean that $\O$ contains $\O'$ and no order lies
strictly between them; the distance between two such orders necessarily divides
$\ell^{2g-1}$ for some prime factor $\ell$ of $[\O_K:\ZZ[\pi,\overline\pi]]$,
since $\O'$ must then contain $\ZZ+\ell\O$.

Recall that Step~4 fails with error probability at most $1/q^{5g^2}$. By
Lemma~\ref{lem:ord-bound}, the number of orders $\O$ considered in Step~3 is at
most $q^{\frac{5}{2}g^2}$, and the number of times this step is reached is
bounded by $\log(q^{\frac{5}{2}g^2})$; the probability that
Algorithm~\ref{alg:ascend} fails is therefore less than $1/q^{g^2}$. We will
later explain how a candidate endomorphism ring may be unconditionally
certified and verified so that, in the unlikely event that
Algorithm~\ref{alg:ascend} does fail, one notices and may start it over again.

\bigskip

The lattice of orders containing $\ZZ[\pi,\overline\pi]$ typically consists
entirely of orders that are either minimal or maximal locally at large primes
$\ell$; indeed, integers $v=[\O_K:\ZZ[\pi,\overline\pi]]$ are not likely to be
divisible by squares of large primes. More precisely, for any $\tau>0$, we have
\[
\#\left\{v\in\{1,\ldots,n\}:\exists\ell\in\P_{>L(n)^\tau},\ell^2|v\right\}
\leq\sum_{\ell\in\P_{>L(n)^\tau}}\frac{n}{\ell^2}\leq\frac{n}{L(n)^\tau},
\]
which is negligible compared to $n$ as it goes to infinity; therefore, assuming
that $v$ has similar divisibility properties to random integers less than
$n=2^{g(2g-1)}q^{g^2/2}$ (as per Lemma~\ref{lem:ord-bound}), only a
zero-density set of abelian varieties of dimension $g$ over $\FF_q$ have
lattices of orders that, locally at some prime $\ell>L(n)^\tau$, have height
greater than $1$.

Discarding that set, there is only one order directly above (resp. below) any
given one locally at large primes $\ell$, and they can be found using a Gröbner
basis algorithm \cite[§\textsc{iii}.2.3]{thesis} in time subexponential in
$\log(q)$. Locally at primes $\ell\leq L(n)^\tau$, we resort to the much more
direct method of enumerating all subgroups of $\frac{1}{\ell}\O/\O$ and
selecting those which are orders; this takes time polynomial in $\ell$, that
is, subexponential in $\log(q)$, and we select $\tau$ small enough so that this
complexity is negligible compared to our overall complexity bound.

Putting all the above together, we obtain:

\begin{hypothesis}[C]
The integers $[\O_K:\ZZ[\pi,\overline\pi]]$ have the divisibility properties
of random integers, so that orders directly above a given one may be enumerated
efficiently as described above.
\end{hypothesis}

\begin{theorem}\label{th:almost}
Subject to Hypotheses (A--C) the expected running time of
Algorithm~\ref{alg:ascend} is bounded by
\[
L(q)^{g^2\sqrt{3}/2+o(1)}.
\]
\end{theorem}

\begin{proof}
The bottleneck of this algorithm is Step~4 which uses
$L(|\disc(\ZZ[\pi,\overline\pi])|)^{g\sqrt{3}/2+o(1)}$ operations by
Proposition~\ref{prop:test}. Using Lemma~\ref{lem:ord-bound}, we may therefore
bound the total complexity by $L(q)^{g^2\sqrt{3}/2+o(1)}$.
\end{proof}

\section{Orders from relation lattices}
\label{sec:orders}

The relation sublattice $\Lambda_\Phi(\O)$ suffices to characterize $\O$
locally at almost all primes:

\def\stuff{\cite[Proposition 20]{quartic-ccc}}
\begin{theorem}[\stuff]\label{th:marco}
Let $(\A_i/\FF_{q_i})_{i\in\NN}$ be a sequence of ordinary abelian varieties
defined over fields of monotonously increasing cardinality $q_i\to\infty$. Denote by
$v_i=[\O_{\QQ(\pi_i)}:\ZZ[\pi_i,\overline\pi_i]]$ their conductor gaps, and by
$n_i=\norm_{K_0/\QQ}(\Delta_{K/K_0})$ the norm of the relative discriminant of
their CM-fields $K=\QQ(\pi_i)$. Assume that there exists a constant $C$ such
that, for all positive integers $u$ and $m$:
\begin{itemize}
\item the proportion of indices $i<m$ for which $u|v_i$ is at most $C/u$;
\item the proportion of indices $i<m$ for which $u|v_i$ and $u|n_i$ is at most $C/u^2$.
\end{itemize}
For any $\tau>0$, the density of indices $i$ for which there exists two orders
$\O$ and $\O'$ containing $\ZZ[\pi_i,\overline\pi_i]$, stable under complex
conjugation, satisfying $\Lambda_\Phi(\O)=\Lambda_\Phi(\O')$, and such that
$\ell^{\val_\ell v_i}>L(q_i)^\tau$ for some prime factor $\ell$ of the index
$[\O+\O':\O\cap\O']$, is zero.
\end{theorem}

The above result essentially means that, excluding a zero-density set of Weil
numbers $\pi$, the lattice $\Lambda_\Phi(\O)$ uniquely characterizes the order
$\O$ from other orders containing $\ZZ[\pi,\overline\pi]$ except locally at
small primes.

Note that the smoothness assumptions are not strong and very similar to that of
Hypothesis~(C): were the $v_i$ drawn uniformly at random from
$\{1,\ldots,(8q_i)^2\}$ and independently from the $n_i$, they would be
satisfied with $C=1$. In practice, they are also found to hold for abelian
varieties to which this work is of most interest: random isomorphism classes of
abelian varieties with complex multiplication by a prescribed field defined
over finite fields of increasing cardinality. At any rate, we assume this
heuristic:

\begin{hypothesis}[D]
The conductor gap and relative discriminant of families of abelian varieties
considered here satisfies the smoothness conditions of the above theorem.
\end{hypothesis}

Under this hypothesis, once an order $\O$ with
$\Lambda_\Phi(\O)=\Lambda_\Phi(\End(\A))$ is found, it only remains to identify
$\End(\A)$ to compute it locally at all prime factors of $\disc(\pi)$ less than
$L(q)^\tau$.

To compute endomorphism rings locally at small primes $\ell$, we rely on the
direct method of Eisenträger and Lauter \cite[§6.5]{eisentrager-lauter}, which
uses $\ell^{2gv+o(1)}$ operations in the base field, where $v$ is the valuation
of $[\O_K:\ZZ[\pi,\overline\pi]]$ at $\ell$. As above, to ensure that this cost
is negligible relative to our overall complexity bound, we make $\tau>0$ small
enough and omit the zero-density set of abelian varieties for which this index
is divisible by a power greater than $L(q)^\tau$ of a prime less than
$L(q)^\tau$. This gives:

\begin{theorem}\label{th:main}
Subject to Hypotheses (A), (B), (C) and (D), the endomorphism ring of an
absolutely simple, ordinary, principally polarized abelian variety of dimension
$g=2$ defined over a finite field with $q$ elements may be computed in
average probabilistic time
\[
L(q)^{g^2\sqrt{3}/2+o(1)}
\]
with error probability less than $1/q^{g^2}$.
\end{theorem}

As an aside, we now describe how one may certify the output endomorphism ring
$\O$ using relations that discriminate $\O$ from other orders of the lattice,
so that it can subsequently be verified solely under Hypothesis~(A).

\begin{definition}
A \emph{certificate} for an order $\O$ consists of:
\begin{itemize}
\item a family of orders $\O_i$ and relations $r_i$ that hold in $\O_i$ but not in $\O$,
\item a family of orders $\O_j$ and relations $r_j$ that hold in $\O$ but not in $\O_j$,
\end{itemize}
such that $\O$ is the only order containing $\ZZ[\pi,\overline\pi]$
satisfying $\O_i\not\subset\O$ and $\O_j\not\supset\O$ for all $i$ and $j$.
\end{definition}

As a direct consequence of Hypothesis~(A), if $\A$ is an absolutely simple,
ordinary, principally polarized abelian variety with Frobenius endomorphism
$\pi$, the relation lattice $\Lambda_\Phi(\End(A))$ and $\Lambda_\Phi(\O)$ are
equal if and only if the isogenies corresponding to the $r_j$'s map $\A$ to
isomorphic varieties while those corresponding to the $r_i$'s do not. In
practice, the $\O_i$'s can be chosen to be all orders considered in Step~3 of
Algorithm~\ref{alg:ascend} for which
$\Lambda_\Phi(\O_i)\not\subset\Lambda_\Phi(\O=\End(\A))$ and the $\O_j$'s to be
all orders directly below $\O$.

A certificate therefore allows anyone to check that the relation lattice of
$\End(\A)$ is that of the claimed endomorphism ring $\O$. By
Propositions~\ref{prop:isog} and \ref{prop:relat}, under Hypotheses (A--D) and
for any $\gamma>0$, it takes $L(q)^{g\gamma+o(1)}+L(q)^{g/(4\gamma)+o(1)}$ time
to generate a certificate that can subsequently be verified under only
Hypotheses (C) in $L(q)^{3g^3\gamma+o(1)}$ operations. Then, to verify that
$\O=\End(\A)$ and not another order with the same relation lattice, this
equality must be verified locally at primes less than $L(q)^\tau$ as per
Theorem~\ref{th:marco}; for $\tau$ small enough, this additional verification
is asymptotically negligible.

\section{Practical Computations}
\label{sec:comp}

We give two examples illustrating different patterns for the index
$v=[\O_K:\ZZ[\pi,\overline\pi]]$. Previous algorithms
\cite{eisentrager-lauter,wagner} compute endomorphism rings efficiently when
$\A[\ell^n]$ remains defined over small extension fields as $\ell^n$ ranges
through prime-power factors of $v$, while ours performs well as soon as no
order directly above $\ZZ[\pi,\overline\pi]$ has an overly large discriminant.

Computations reported here were performed by a straightforward Magma
\cite{magma} implementation using the AVIsogenies library \cite{avisogenies}
and running on one Intel i7-2620M core.

\subsection{Example with nearly prime {\em v}}

Let us first consider a very favorable case where $v$ is both large and nearly
prime, that of the Jacobian variety $\A$ of the hyperelliptic curve with
equation
\[
y^2 = x^5 + 523747 x^4 + 306186 x^3 + 744660 x^2 + 415524 x + 261884
\]
over the field with $q=1250407$ elements; its Frobenius endomorphism $\pi$
admits the characteristic polynomial $z^4 + 1251 z^3 + 1772074 z^2 + 1251 q z +
q^2$ from which one can derive that $\ZZ[\pi,\overline\pi]$ is an order of
index $v=2\cdot 538259$ in the ring of integers of $K=\QQ(\pi)$.

We start by computing $\End(\A)$ locally at $2$, that is, determining whether
it contains the order in which $\ZZ[\pi,\overline\pi]$ has index $2$; this
order is generated by $\pi$ and $\alpha/(2q)$ where
\[
\alpha=417q + 1346084914086\pi + 497115559392\pi^2 + \pi^3.
\]
To determine whether $\alpha/(2q)$ belongs to $\End(\A)$ or, equivalently,
whether $\alpha/2$ does (as $q$ is coprime to $v$), we use the method of
Eisenträger and Lauter \cite{eisentrager-lauter}: it takes $102$ms to determine
that $\alpha$ kills the full $2$-torsion of $\A$, which establishes that
$\End(\A)$ is locally maximal at $2$.

Now denote by $\ppp\overline\ppp$ the factorization of $7$ in
$\ZZ[\pi,\overline\pi]$ and observe that $\ppp$ is principal in $\O_K$. We
evaluate the corresponding isogeny, spending $10.9$s to find its kernel and
$1.37$s to identify the isogenous variety; since it is not isomorphic to $\A$
we have established, in just $12.3$s, that 
\[
\End(\A)\simeq \ZZ[\pi,\alpha/(2q)].
\]
This computation is clearly intractable using previous algorithms: the full
$538259$-torsion of $\A$ is defined over an extension of degree
$e=869166638466$, so it would require a rough minimum of
$\log(q^e)\log(q^{eg})\approx 2^{90}$ operations just to find a random
$538259$-torsion point.

\subsection{Example with composite {\em v}}

For a less degenerate case, let $\A$ be the Jacobian variety of the curve with
equation
\[
y^2 = x^5 + 800 x^4 + 2471 x^3 + 6695 x^2 + 1082 x + 7062
\]
over the field with $q=7681$ elements. It takes just $60$ms to compute that the
characteristic polynomial of its Frobenius endomorphism is $z^4 + 114 z^3 +
7566 z^2 + 114 q z + q^2$ from which it takes negligible time to derive that
$\ZZ[\pi,\overline\pi]$ has index $2^2\cdot 47^2\cdot 379$ in $\O_K$.

Again, we start by computing the endomorphism ring locally at $2$ using the
method of Eisenträger and Lauter \cite{eisentrager-lauter}. Only $75$ms are
needed to find a basis for the full $2$-torsion (the $4$-torsion is not needed)
and evaluate the relevant endomorphism on it; this determined that $\End(\A)$
contains the order $\O_2=\ZZ[\pi,\overline\pi]+47^2\cdot 379\cdot\O_K$. Having
established that, we may start Algorithm~\ref{alg:ascend} from the order $\O_2$
instead of $\ZZ[\pi,\overline\pi]$; the two orders directly above $\O_2$ have
index $379$ and $47^2$ in $\O_K$.

First consider that of index $47^2$: in just $100$ms we find that ideals of
norm $3^2$ have order $92$ in its class group. Computing the $92$ corresponding
isogenies takes $37$s, that is, $400$ms on average. As the isogenous variety is
not isomorphic to $\A$, we deduce that $\End(\A)$ is minimal locally at $47$.

Next we consider the order with index $379$; after $150$ms, we find that the
ideal $\ppp^{62}(\rrr\sss)^{2}$ is principal in it, where the primes appear in
the splittings $3=\ppp\overline\ppp$ and
$19=\rrr\sss\overline\rrr\overline\sss$. We therefore proceed to test whether
the corresponding relation holds in $\A$: it takes $67$s on average to compute
each of the two $19$-isogenies, and $400$ms for each of the $3$-isogenies. The
isogenous variety, which is determined after a total of $157$s, is not found to
be isomorphic to $\A$, hence we deduce that $\End(\A)$ is the order containing
$\ZZ[\pi,\overline\pi]$ with index $4$.

Note that the full $47$-torsion and full $379$-torsion live over extensions of
degree $34592$ and $13609890$ respectively, which again makes computing
$\End(\A)$ using previous methods quite expensive.

\bigskip

This illustrates that, even when the orders in which we look for relations have
moderate class numbers, the bottleneck of our algorithm remains the evaluation
of isogenies. Accordingly, in both computations above, we have used a simple
baby-step giant-step method in place of Algorithm~\ref{alg:relat}, which
allowed us to find much smaller relations and therefore to better balance the
cost of evaluating isogenies with that of searching for relations.

Overall, we find that our algorithm clearly outperforms previous methods as
soon as the index $[\O_K:\ZZ[\pi,\overline\pi]]$ has prime power factors
$\ell^n$ for which the torsion points live over significant extensions of the
base field, although those methods are still very useful to compute the
endomorphism ring locally at small primes.

\section*{Acknowledgments}

This work would never have seen the light of day without the author's prior
collaborations with Andrew V. Sutherland, constant encouragements from Pierrick
Gaudry, and invaluable discussions with Andreas Enge, Igor Shparlinski, and
Marco Streng.

\printbibliography

\end{document}